\documentclass[leqno,12pt]{amsart}
\usepackage{amssymb}
\usepackage{amsmath}
\usepackage{enumerate}
\usepackage{amsfonts}
\usepackage{hyperref}
\usepackage{tikz}

\definecolor{darkgreen}{RGB}{45, 119, 75}

\newcommand{\supp}{\text{supp }}

\newtheorem{theorem}{Theorem}[section]
\newtheorem{corollary}[theorem]{Corollary}
\newtheorem{lemma}[theorem]{Lemma}
\newtheorem{proposition}[theorem]{Proposition}

\newtheorem{definition}[theorem]{Definition}
\newtheorem{fact}[theorem]{Fact}

\numberwithin{equation}{section}

\hypersetup{
    colorlinks=true,
    linkcolor= blue,
    citecolor =cyan,
    urlcolor = teal,
}

\begin{document}

\title[Behavior of eigenvalues]{Behavior of eigenvalues \\\ of certain Schr\"odinger operators \\ in the rational Dunkl setting}

\subjclass[2010]{{primary: 35P20; secondary: 35J10,  42B37, 35K08, 42B35.}}
\keywords{Rational Dunkl theory, Schr\"odinger operators, Asymptotic distributions of eigenvalues, Reverse H\"older classes, Fefferman--Phong inequality.}

\author[Agnieszka Hejna]{Agnieszka Hejna}
\begin{abstract}
  For a normalized root system $R$ in $\mathbb R^N$ and a multiplicity function $k\geq 0$ let $\mathbf N=N+\sum_{\alpha \in R} k(\alpha)$. We denote by $dw(\mathbf{x})=\Pi_{\alpha \in R}|\langle \mathbf{x},\alpha \rangle|^{k(\alpha)}\,d\mathbf{x}$ the associated measure in $\mathbb{R}^N$.
Let $L=-\Delta +V$, $V\geq 0$,  be the Dunkl--Schr\"odinger operator on $\mathbb R^N$. Assume that there exists $q >\max(1,\frac{\mathbf{N}}{2})$ such that  $V$ belongs to the reverse H\"older class ${\rm{RH}}^{q}(dw)$. For $\lambda>0$ we provide  upper and lower estimates for the number of eigenvalues of $L$ which are less or equal to $\lambda$. Our main tool in the Fefferman--Phong type inequality in the rational Dunkl setting.
\end{abstract}

\address{Agnieszka Hejna, Uniwersytet Wroc\l awski,
Instytut Matematyczny,
Pl. Grunwaldzki 2/4,
50-384 Wroc\l aw,
Poland}
\email{hejna@math.uni.wroc.pl}

\thanks{
Research supported by the National Science Centre, Poland (Narodowe Centrum Nauki), Grant 2017/25/B/ST1/00599}

\maketitle

\section{Introduction}
On $\mathbb R^N$ equipped with a normalized root system $R$ and a multiplicity function $k \geq 0$, we consider a Schr\"odinger operator 
$$ L=-\Delta +V(\mathbf{x}),$$
where $\Delta$  is the Dunkl Laplacian and $V$ is a non-negative function which satisfies the reverse H\"older inequality: 
\begin{equation}\label{eq:RevHol} \Big( \frac{1}{w(B)} \int_B V(\mathbf y)^q\, dw(\mathbf y)\Big)^{1/q}\leq \frac{C_{\rm RH}}{w(B)}\int_B V(\mathbf y)\, dw(\mathbf y) \quad \text{\rm for every ball }B.
\end{equation}
Here and subsequently, 
\begin{equation}\label{eq:measure_formula}
dw(\mathbf x)=\prod_{\alpha\in R}|\langle \mathbf x,\alpha\rangle|^{k(\alpha)}\, d\mathbf x
\end{equation} 
is the associated measure. We shall assume that $q>\max (1,\frac{\mathbf N}{2})$, where $\mathbf N=N+\sum_{\alpha\in R} k(\alpha)$ is the homogeneous dimension. Following \cite[p.  146, the assumption of the main lemma]{Fefferman} (see also~\cite[(4.1)]{Hejna} for its counterpart in the rational Dunkl setting) we define the auxiliary function $m(\mathbf x)$ by the formula 
\begin{equation}\label{eq:m}
    \frac{1}{m(\mathbf{x})}=\sup\left\{r>0\,:\; \frac{r^2}{w(B(\mathbf{x},r))}\int_{B(\mathbf{x},r)}V(\mathbf{y})\,dw(\mathbf{y}) \leq 1 \right\}.
\end{equation}
The function $m$ is well defined and satisfies $0<m(\mathbf x)<\infty$ for every $\mathbf x\in\mathbb R^N$, see~\cite{Hejna}.  

For $\lambda>0$, we denote by $N(L,\lambda)$ the number of eigenvalues of the operator  $L$, counting with their multiplicities, which are  less than or  equal to $\lambda$. 

For $a>0$ we define
\begin{equation}\label{eq:grid}
    {({\rm Grid})_{a}}= \{[0,a]^N+a\mathbf{n}\;:\;\mathbf{n} \in \mathbb{Z}^{N}\}.
\end{equation}
Our goal is to prove the following theorem. 
\begin{theorem}\label{teo:main_box}
Assume that $V \in {\rm{RH}}^{q}(dw)$, where $q>\max(1,\frac{\mathbf{N}}{2})$, and $V \geq 0$. For $\lambda>0$ we set 
\begin{align*}
    E_{\lambda}=\{\mathbf{x} \in \mathbb{R}^N\;:\; m(\mathbf{x}) \leq \sqrt{\lambda}\}.
\end{align*}
Let $M(\lambda)$ denote the number of cubes $K$ from the {$({\rm Grid})_{\lambda^{-1/2}}$}  (see~\eqref{eq:grid}) such that $K \cap E_{\lambda} \neq \emptyset$. There are constants $C_1,C_2,C_3>0$, which depend on $R$, $N$, $q$, $k$ and the constant $C_{\rm RH}$ (see \eqref{eq:RevHol}) such that for all $\lambda>0$ we have
\begin{equation}\label{eq:main_1}
    M(C_1^{-1}\lambda) \leq N(L,\lambda) \leq C_2M(C_3^{-1}\lambda).
\end{equation}
Let $\lambda_0(L)$ denote  the smallest eigenvalue of $L$. There is a constant $C_4>0$, which depends on $R$, $N$, $q$, $k$ and the constant $C_{\rm RH}$, such that
\begin{equation}\label{eq:min_m}
    \lambda_{0}(L) \geq C_4 \min_{\mathbf{x} \in \mathbb{R}^N}m(\mathbf{x}).
\end{equation}
Actually, one can take $C_4=(\widetilde{C})^{1/2}$, where $\widetilde{C}$ is the constant from~\eqref{eq:F-Ph}.
\end{theorem}

For classical Schr\"odinger operators with reverse H\"older class potentials on $\mathbb R^N$ behavior of eigenvalues were studied in the seminal article Feffreman \cite{Fefferman} and then continued by many authors (see e.g.\cite{KurataSugano},~\cite{Shen3},~\cite{Shen4},~\cite{Simon},~\cite{Tachizawa}).  The present article takes inspirations from there. 

\section{Preliminaries and notation}

In this section we present necessary definitions and lemmas (with references), which will be used in the proof of Theorem~\ref{teo:main_box}.

\subsection{Basic definitions of the Dunkl theory}

In this section we present basic facts concerning the theory of the Dunkl operators.  For details we refer the reader to~\cite{Dunkl},~\cite{Roesler3}, and~\cite{Roesler-Voit}. 

We consider the Euclidean space $\mathbb R^N$ with the scalar product $\langle\mathbf x,\mathbf y\rangle=\sum_{j=1}^N x_jy_j
$, where $\mathbf x=(x_1,...,x_N)$, $\mathbf y=(y_1,...,y_N)$, and the norm $\| \mathbf x\|^2=\langle \mathbf x,\mathbf x\rangle$. For a nonzero vector $\alpha\in\mathbb R^N$,  the reflection $\sigma_\alpha$ with respect to the hyperplane $\alpha^\perp$ orthogonal to $\alpha$ is given by
\begin{align*}
\sigma_\alpha (\mathbf x)=\mathbf x-2\frac{\langle \mathbf x,\alpha\rangle}{\| \alpha\| ^2}\alpha.
\end{align*}
In this paper we fix a normalized root system in $\mathbb R^N$, that is, a finite set  $R\subset \mathbb R^N\setminus\{0\}$ such that $R \cap \alpha \mathbb{R} = \{\pm \alpha\}$,  $\sigma_\alpha (R)=R$, and $\|\alpha\|=\sqrt{2}$ for all $\alpha\in R$. The finite group $G$ generated by the reflections $\sigma_\alpha \in R$ is called the {\it Weyl group} ({\it reflection group}) of the root system. A~{\textit{multiplicity function}} is a $G$-invariant function $k:R\to\mathbb C$ which will be fixed and $\geq 0$  throughout this paper. 
 Let
\begin{equation*}
dw(\mathbf x)=\prod_{\alpha\in R}|\langle \mathbf x,\alpha\rangle|^{k(\alpha)}\, d\mathbf x
\end{equation*} 
be  the associated measure in $\mathbb R^N$, where, here and subsequently, $d\mathbf x$ stands for the Lebesgue measure in $\mathbb R^N$.
We denote by 
\begin{align*}
\mathbf{N}=N+\sum_{\alpha \in R} k(\alpha)
\end{align*}
the homogeneous dimension of the system. Clearly, 
\begin{align*} w(B(t\mathbf x, tr))=t^{\mathbf N}w(B(\mathbf x,r)) \ \ \text{\rm for all } \mathbf x\in\mathbb R^N, \ t,r>0,   
\end{align*}
 where $B(\mathbf x, r)=\{\mathbf y\in\mathbb R^N: \|\mathbf y-\mathbf x\|<r\}$. Moreover, 
\begin{align*}
\int_{\mathbb R^N} f(\mathbf x)\, dw(\mathbf x)=\int_{\mathbb R^N} t^{-\mathbf N} f(\mathbf x\slash t)\, dw(\mathbf x)\ \ \text{for} \ f\in L^1(dw)  \   \text{\rm and} \  t>0.
\end{align*}
Observe that there is a constant $C>0$ such that 
\begin{equation}\label{eq:balls_asymp} 
C^{-1}w(B(\mathbf x,r))\leq  r^{N}\prod_{\alpha \in R} (|\langle \mathbf x,\alpha\rangle |+r)^{k(\alpha)}\leq C w(B(\mathbf x,r)),
\end{equation}
so $dw(\mathbf x)$ is doubling, that is, there is a constant $C>0$ such that
\begin{equation}\label{eq:doubling} w(B(\mathbf x,2r))\leq C w(B(\mathbf x,r)) \ \ \text{ for all } \mathbf x\in\mathbb R^N, \ r>0.
\end{equation}
Moreover, there exists a constant $C\ge1$ such that,
for every $\mathbf{x}\in\mathbb{R}^N$ and for every $r_2\ge r_1>0$,
\begin{equation}\label{eq:growth}
C^{-1}\Big(\frac{r_2}{r_1}\Big)^{N}\leq\frac{{w}(B(\mathbf{x},r_2))}{{w}(B(\mathbf{x},r_1))}\leq C \Big(\frac{r_2}{r_1}\Big)^{\mathbf{N}}.
\end{equation}

For a measurable subset $A$ of $\mathbb{R}^N$ we define 
\begin{align*}
    \mathcal{O}(A)=\{\sigma_{\alpha}(\mathbf{x})\,:\, \mathbf{x} \in A, \, \alpha \in R\}.
\end{align*}
Clearly, by~\eqref{eq:balls_asymp}, for all $\mathbf{x} \in \mathbb{R}^N$ and $r>0$ we get
\begin{align*}
    w(\mathcal{O}(B(\mathbf{x},r))) \leq |G|w(B(\mathbf{x},r)).
\end{align*}

For $\xi \in \mathbb{R}^N$, the {\it Dunkl operators} $T_\xi$  are the following $k$-deformations of the directional derivatives $\partial_\xi$ by a  difference operator:
\begin{align*}
     T_\xi f(\mathbf x)= \partial_\xi f(\mathbf x) + \sum_{\alpha\in R} \frac{k(\alpha)}{2}\langle\alpha ,\xi\rangle\frac{f(\mathbf x)-f(\sigma_\alpha(\mathbf{x}))}{\langle \alpha,\mathbf x\rangle}.
\end{align*}
The Dunkl operators $T_{\xi}$, which were introduced in~\cite{Dunkl}, commute and are skew-symmetric with respect to the $G$-invariant measure $dw$.

For fixed $\mathbf y\in\mathbb R^N$ the {\it Dunkl kernel} $E(\mathbf x,\mathbf y)$ is the unique analytic solution to the system
\begin{align*}
    T_\xi f=\langle \xi,\mathbf y\rangle f, \ \ f(0)=1.
\end{align*}
The function $E(\mathbf x ,\mathbf y)$, which generalizes the exponential  function $e^{\langle \mathbf x,\mathbf y\rangle}$, has the unique extension to a holomorphic function on $\mathbb C^N\times \mathbb C^N$. Moreover, it satisfies $E(\mathbf{x},\mathbf{y})=E(\mathbf{y},\mathbf{x})$ for all $\mathbf{x},\mathbf{y} \in \mathbb{C}^N$.

Let $\{e_j\}_{1 \leq j \leq N}$ denote the canonical orthonormal basis in $\mathbb R^N$ and let $T_j=T_{e_j}$. As usual, for every multi-index \hspace{.5mm}$\alpha\hspace{-.5mm}=\hspace{-.5mm}(\alpha_1,\alpha_2,\dots,\alpha_N)\!\in\hspace{-.5mm}\mathbb{N}_0^N=(\mathbb{N} \cup \{0\})^{N}$,
we set \hspace{.25mm}$|\alpha|\!
=\hspace{-.5mm}\sum_{\hspace{.25mm}j=1}^{\hspace{.5mm}N}\hspace{-.25mm}\alpha_j$
and
$$
\partial^{\hspace{.25mm}\alpha}\!
=\partial_{e_1}^{\hspace{.25mm}\alpha_1}\!\circ\hspace{-.25mm}
\partial_{e_2}^{\hspace{.25mm}\alpha_{\hspace{.2mm}2}}
\!\circ\ldots\circ\hspace{-.25mm}
\partial_{e_{N}}^{\hspace{.25mm}\alpha_{\hspace{-.2mm}N}}\,,
$$
where $\{e_1,e_2\hspace{.25mm},\ldots,e_{N}\}$ is the canonical basis of $\mathbb{R}^N$.
The additional subscript $\mathbf{x}$ in $\partial^{\hspace{.25mm}\alpha}_{\mathbf{x}}$
means that the partial derivative \hspace{.25mm}$\partial^{\hspace{.25mm}\alpha}$
is taken with respect to the variable $\mathbf{x}\!\in\!\mathbb{R}^N$. By $\nabla_{\mathbf{x}}f$ we denote the gradient of the function $f$ with respect to the variable $\mathbf{x}$. 
In our further consideration we shall need the following lemma.
\begin{lemma}
For all $\mathbf{x} \in \mathbb{R}^N$, $\mathbf{z} \in \mathbb{C}^N$ and $\nu \in \mathbb{N}_0^{N}$ we have
$$|\partial^{\nu}_{\mathbf{z}}E(\mathbf{x},\mathbf{z})| \leq \|\mathbf{x}\|^{|\nu|}\exp(\|\mathbf{x}\|\|{\rm Re \;}\mathbf{z}\|).$$
In particular, \begin{align*} | E(i\xi, \mathbf x)|\leq 1 \quad \text{ for all } \xi,\mathbf x\in \mathbb R^N.
\end{align*}
\end{lemma}
\begin{proof}
See~\cite[Corollary 5.3]{Roesle99}.
\end{proof}

\begin{corollary}\label{coro:Roesler}
There is a constant $C>0$ such that for all $\mathbf{x},\xi \in \mathbb{R}^N$ we have
\begin{equation}
|E(i\xi,\mathbf{x})-1| \leq C\|\mathbf{x}\|\|\xi\|.
\end{equation}
\end{corollary}

The \textit{Dunkl transform}
  \begin{align*}\mathcal F f(\xi)=c_k^{-1}\int_{\mathbb R^N} E(-i\xi, \mathbf x)f(\mathbf x)\, dw(\mathbf x),
  \end{align*}
  where
  $$c_k=\int_{\mathbb{R}^N}e^{-\frac{\|\mathbf{x}\|^2}{2}}\,dw(\mathbf{x})>0,$$
   originally defined for $f\in L^1(dw)$, is an isometry on $L^2(dw)$, i.e.,
   \begin{equation}\label{eq:Plancherel}
       \|f\|_{L^2(dw)}=\|\mathcal{F}f\|_{L^2(dw)} \text{ for all }f \in L^2(dw),
   \end{equation}
and preserves the Schwartz class of functions $\mathcal S(\mathbb R^N)$ (see \cite{deJeu}). Its inverse $\mathcal F^{-1}$ has the form
  \begin{align*} \mathcal F^{-1} g(\mathbf{x})=c_k^{-1}\int_{\mathbb R^N} E(i\xi, \mathbf x)g(\xi)\, dw(\xi).
  \end{align*}
Moreover,
\begin{equation}\label{eq:der_transform}
    \mathcal{F}(T_j f)(\xi)=i\xi_j\mathcal{F}f(\xi).
\end{equation}
The {\it Dunkl translation\/} $\tau_{\mathbf{x}}f$ of a function $f\in\mathcal{S}(\mathbb{R}^N)$ by $\mathbf{x}\in\mathbb{R}^N$ is defined by
\begin{align*}
\tau_{\mathbf{x}} f(\mathbf{y})=c_k^{-1} \int_{\mathbb{R}^N}{E}(i\xi,\mathbf{x})\,{E}(i\xi,\mathbf{y})\,\mathcal{F}f(\xi)\,{dw}(\xi).
\end{align*}
  It is a contraction on $L^2(dw)$, however it is an open  problem  if the Dunkl translations are bounded operators on $L^p(dw)$ for $p\ne 2$.
  
 The following specific formula was obtained by R\"osler \cite{Roesler2003} for the Dunkl translations of (reasonable) radial functions $f({\mathbf{x}})=\tilde{f}({\|\mathbf{x}\|})$:
\begin{equation}\label{eq:translation-radial}
\tau_{\mathbf{x}}f(-\mathbf{y})=\int_{\mathbb{R}^N}{(\tilde{f}\circ A)}(\mathbf{x},\mathbf{y},\eta)\,d\mu_{\mathbf{x}}(\eta)\text{ for all }\mathbf{x},\mathbf{y}\in\mathbb{R}^N.
\end{equation}
Here
\begin{equation*}
A(\mathbf{x},\mathbf{y},\eta)=\sqrt{{\|}\mathbf{x}{\|}^2+{\|}\mathbf{y}{\|}^2-2\langle \mathbf{y},\eta\rangle}=\sqrt{{\|}\mathbf{x}{\|}^2-{\|}\eta{\|}^2+{\|}\mathbf{y}-\eta{\|}^2}
\end{equation*}
and $\mu_{\mathbf{x}}$ is a probability measure, 
which is supported in the set $\operatorname{conv}\mathcal{O}(\mathbf{x})$,  where $\mathcal O(\mathbf x) =\{\sigma(\mathbf x): \sigma \in G\}$ is the orbit of $\mathbf x$. Let 
$$d(\mathbf x,\mathbf y)=\min_{\sigma\in G}\| \sigma(\mathbf x)-\mathbf y\|$$
be the distance of the orbit of $\mathbf x$ to the orbit of $\mathbf y$. We have the following  elementary estimates (see, e.g., \cite{AS}){, which hold for $\mathbf{x},\mathbf{y}\in\mathbb{R}^N$ and $\eta\in\operatorname{conv}\mathcal{O}(\mathbf{x})$\,:
\begin{align*}
A(\mathbf{x},\mathbf{y},\eta)\ge d(\mathbf{x},\mathbf{y})
\end{align*}
and}
\begin{align*}\begin{cases}
\,{\|}\nabla_{\mathbf{y}}\{A(\mathbf{x},\mathbf{y},\eta)^2\}{\|}\le{2}\,A(\mathbf{x},\mathbf{y},\eta),\\
\,|\hspace{.25mm}\partial^\beta_{\mathbf{y}}\{A(\mathbf{x},\mathbf{y},\eta)^2\}|\le{2}
&{\text{if \,}|\beta|=2\hspace{.25mm},}\\
\,\partial^\beta_{\mathbf{y}}\{A(\mathbf{x},\mathbf{y},\eta)^2\}=0
&{\text{if \,}|\beta|>2\hspace{.25mm}.}
\end{cases}\end{align*}
Hence
\begin{equation}\label{A3}
{\|}\nabla_{\mathbf{y}}A(\mathbf{x},\mathbf{y},\eta){\|}\le{1}.
\end{equation}
  
  {The \textit{Dunkl convolution\/} $f*g$ of two reasonable functions (for instance Schwartz functions) is defined by
\begin{equation}\label{eq:conv_def}
(f*g)(\mathbf{x})=c_k\,\mathcal{F}^{-1}[(\mathcal{F}f)(\mathcal{F}g)](\mathbf{x})=\int_{\mathbb{R}^N}(\mathcal{F}f)(\xi)\,(\mathcal{F}g)(\xi)\,E(\mathbf{x},i\xi)\,dw(\xi) \text{ for }\mathbf{x}\in\mathbb{R}^N,
\end{equation}
or, equivalently, by}
\begin{align*}
  {(f{*}g)(\mathbf{x})=\int_{\mathbb{R}^N}f(\mathbf{y})\,\tau_{\mathbf{x}}g(-\mathbf{y})\,{dw}(\mathbf{y})=\int_{\mathbb R^N} f(\mathbf y)g(\mathbf x,\mathbf y) \,dw(\mathbf{y}) \text{ for all } \mathbf{x}\in\mathbb{R}^N},  
\end{align*}
where, here and subsequently, $g(\mathbf x,\mathbf y)=\tau_{\mathbf x}g(-\mathbf y)$. 

\subsection{Dunkl Laplacian and Dunkl heat semigroup} The {\it Dunkl Laplacian} associated with $R$ and $k$  is the differential-difference operator $\Delta=\sum_{j=1}^N T_{j}^2$, which  acts on $C^2(\mathbb{R}^N)$-functions by\begin{align*}
    \Delta f(\mathbf x)=\Delta_{\rm eucl} f(\mathbf x)+\sum_{\alpha\in R} k(\alpha) \delta_\alpha f(\mathbf x),
\end{align*}
\begin{align*}
    \delta_\alpha f(\mathbf x)=\frac{\partial_\alpha f(\mathbf x)}{\langle \alpha , \mathbf x\rangle} - \frac{\|\alpha\|^2}{2} \frac{f(\mathbf x)-f(\sigma_\alpha \mathbf x)}{\langle \alpha, \mathbf x\rangle^2}.
\end{align*}
Obviously, $\mathcal F(\Delta f)(\xi)=-\| \xi\|^2\mathcal Ff(\xi)$. The operator $\Delta$ is essentially self-adjoint on $L^2(dw)$ (see for instance \cite[Theorem\;3.1]{AH}) and generates the semigroup $H_t$  of linear self-adjoint contractions on $L^2(dw)$. The semigroup has the form
  \begin{align*}
  H_t f(\mathbf x)=\mathcal F^{-1}(e^{-t\|\xi\|^2}\mathcal Ff(\xi))(\mathbf x)=\int_{\mathbb R^N} h_t(\mathbf x,\mathbf y)f(\mathbf y)\, dw(\mathbf y),
  \end{align*}
  where the heat kernel 
  \begin{align*}
      h_t(\mathbf x,\mathbf y)=\tau_{\mathbf x}h_t(-\mathbf y), \ \ h_t(\mathbf x)=\mathcal F^{-1} (e^{-t\|\xi\|^2})(\mathbf x)=c_k^{-1} (2t)^{-\mathbf N\slash 2}e^{-\| \mathbf x\|^2\slash (4t)}
  \end{align*}
  is a $C^\infty$-function of all variables $\mathbf x,\mathbf y \in \mathbb{R}^N$, $t>0$, and satisfies \begin{align*} 0<h_t(\mathbf x,\mathbf y)=h_t(\mathbf y,\mathbf x),
  \end{align*}
 \begin{align*}
    \int_{\mathbb R^N} h_t(\mathbf x,\mathbf y)\, dw(\mathbf y)=1.
 \end{align*}
  We shall need the following estimates for $h_t(\mathbf x,\mathbf y)$ - their two step proof, which is based on R\"osler's formula~\eqref{eq:translation-radial} for the Dunkl  translations of  radial functions (see~\cite{Roesler2003}), can be found in~\cite[Theorem 4.1]{ADzH} and~\cite[Theorem 3.1]{DzH1}. 
\begin{theorem}\label{theorem:heat}
There are constants $C,c>0$ such that for all $\mathbf{x},\mathbf{y} \in \mathbb{R}^N$ and $t>0$ we have
\begin{align*}
h_t(\mathbf{x},\mathbf{y}) \leq C\Big(1+\frac{\|\mathbf{x}-\mathbf{y}\|}{\sqrt{t}}\Big)^{-2}\Big(\max (w(B(\mathbf x,\sqrt{t})),w(B(\mathbf y, \sqrt{t})))\Big)^{-1}\exp\Big(-\frac {cd(\mathbf x,\mathbf y)^2}{t}\Big).
\end{align*}
\end{theorem}

Theorem~\ref{theorem:heat} implies the following Lemma (see~\cite[Corollary 3.5]{DzH1}).

\begin{lemma}\label{lem:nonradial_estimation} 
Suppose that $\varphi \in C_c^{\infty}(\mathbb{R}^N)$ is radial and supported by the unit ball $B(0,1)$. Set $\varphi_t(\mathbf{x})=t^{-\mathbf{N}}\varphi(t^{-1}\mathbf{x})$. Then there is $C>0$ such that for all $\mathbf{x},\mathbf{y}\in \mathbb{R}^{N}$ and $t>0$ we have
\begin{align*}
    &|\varphi_t(\mathbf{x},\mathbf{y})| \leq C \Big(1+\frac{\|\mathbf{x}-\mathbf{y}\|}{t}\Big)^{-2}\Big(\max(w(B(\mathbf{x},t)),w(B(\mathbf{y},t)))\Big)^{-1}\chi_{[0,1]}(d(\mathbf{x},\mathbf{y})/t).
\end{align*}
\end{lemma}

\subsection{Dunkl-Schr\"odinger operator}\label{sec:Schorodinger}
Let $V \geq 0$ be a measurable function such that $V \in L^2_{\rm{loc}}(dw)$. We consider the following operator on the Hilbert space $L^2(dw)$:
\begin{align*}
    \mathcal{L}=-\Delta+V
\end{align*}
with the domain
\begin{align*}
    \mathcal{D}(\mathcal{L})=\{f \in L^2(dw)\,:\,\|\xi\|^2\mathcal{F}f(\xi) \in L^2(dw(\xi)) \text{ and }V(\mathbf{x})f(\mathbf{x}) \in L^2(dw(\mathbf{x}))\}
\end{align*}
(see~\cite{AH}). We call this operator the \textit{Dunkl-Schr\"odinger operator}. Let us define the quadratic form
\begin{align*}
    \mathbf{Q}(f,g)=\sum_{j=1}^{N}\int_{\mathbb{R}^N}T_jf(\mathbf{x})\overline{T_jg(\mathbf{x})}\,dw(\mathbf{x})+\int_{\mathbb{R}^N}V(\mathbf{x})f(\mathbf{x})\overline{g(\mathbf{x})}\,dw(\mathbf{x})
\end{align*}
with the domain
\begin{align*}
    \mathcal{D}(\mathbf{Q})=\left\{f \in L^2(dw)\;:\; \left(\sum_{j=1}^{N}|T_jf|^2\right)^{1/2},V^{1/2}f \in L^2(dw)\right\} .
\end{align*}
The quadratic form is densely defined and closed (see~\cite[Lemma 4.1]{AH}), so there exists a unique positive self-adjoint operator $L$ such that
\begin{align*}
    \langle Lf,f\rangle=\mathbf{Q}(f,f) \text{ for all }f \in \mathcal{D}(L), 
\end{align*}
moreover,
\begin{align*}
    \mathcal{D}(L^{1/2})=\mathcal{D}(\mathbf{Q}) \text{ and }\mathbf{Q}(f,f)=\|L^{1/2}f\|_{L^2(dw)},
\end{align*}
where $L^{1/2}$ is a unique self-adjoint operator such that $(L^{1/2})^2=L$. It was proved in~\cite[Theorem 4.6]{AH}, that $\mathcal{L}$ is essentially self-adjoint on $C^{\infty}_{c}(\mathbb{R}^N)$ and $L$ is its closure.

\subsection{Auxiliary function \texorpdfstring{$m$}{m} and the Fefferman-Phong type inequality}
The results in this subsection are proved in~\cite{Hejna} and they will be used in the proof of Theorem~\ref{teo:main_box}. Some of them are inspired by the corresponding results for classical Schr\"odinger operators (cf. ~\cite{Fefferman} and~\cite{Shen}).

\begin{lemma}[{\cite[Lemma 3.8]{Hejna}}, see also{~\cite[Lemma 1.2]{Shen}}]\label{lem:Rr}
Assume that $V \in {\rm{RH}}^{q}(dw)$, where $q>\max(1,\frac{\mathbf{N}}{2})$, and $V \geq 0$. There is a constant $C \geq 1$ such that for all $\mathbf{x} \in \mathbb{R}^N$ and $0<r_1<r_2<\infty$ we have
\begin{align*}
    \frac{r_1^2}{w(B(\mathbf{x},r_1))}\int_{B(\mathbf{x},r_1)}V(\mathbf{y})\,dw(\mathbf{y}) \leq C \left(\frac{r_1}{r_2}\right)^{\gamma} \frac{r_2^2}{w(B(\mathbf{x},r_2))}\int_{B(\mathbf{x},r_2)}V(\mathbf{y})\,dw(\mathbf{y}). 
\end{align*}
\end{lemma}

\begin{lemma}[{\cite[Lemma 4.1]{Hejna}}, see also{~\cite[Lemma 1.4]{Shen}}]\label{lem:m_growth}
Assume that $V \in {\rm{RH}}^{q}(dw)$, where $q>\max(1,\frac{\mathbf{N}}{2})$, and $V \geq 0$. There are constants $C,\kappa>0$ such that for all $\mathbf{x},\mathbf{y} \in \mathbb{R}^N$ we have
\begin{align*}
    C^{-1}m(\mathbf{y}) \leq m(\mathbf{x}) \leq Cm(\mathbf{y}) \text { if }\|\mathbf{x}-\mathbf{y}\|<m(\mathbf{x})^{-1},
\end{align*}
\begin{align*}
    m(\mathbf{y}) \leq Cm(\mathbf{x})(1+m(\mathbf{x})\|\mathbf{x}-\mathbf{y}\|)^{\kappa},
\end{align*}
\begin{align*}
    m(\mathbf{y}) \geq C^{-1}m(\mathbf{x})(1+m(\mathbf{x})\|\mathbf{x}-\mathbf{y}\|)^{-\frac{\kappa}{1+\kappa}}.
\end{align*}
\end{lemma}

For a cube $Q \subset \mathbb{R}^N$, here and subsequently, let $d(Q)$ denote the side-length of cube $Q$. We denote by $Q^{*}$ the cube with the same center as $Q$ such that $d(Q^{*})=2d(Q)$. We define a collection of dyadic cubes $\mathcal{Q}$ associated with the potential $V$ by the following stopping-time condition:
\begin{equation}\label{eq:stopping_time}
    Q \in \mathcal{Q} \iff Q \text{ is the maximal dyadic cube for which }\frac{d(Q)^2}{w(Q)}\int_{Q} V(\mathbf{y})\,dw(\mathbf{y}) \leq 1
\end{equation}
(see~\cite[(4.7)]{Hejna}). It is well-defined (see the comment below~\cite[(4.7)]{Hejna} for details) and it forms a covering of $\mathbb{R}^N$ built from dyadic cubes which have disjoint interiors.

\begin{fact}[{\cite[Fact 4.3]{Hejna}}]\label{fact:m_d_compare}
Assume that $V \in {\rm{RH}}^{q}(dw)$, where $q>\max(1,\frac{\mathbf{N}}{2})$, and $V \geq 0$. There is a constant $C>0$ such that for any $Q \in \mathcal{Q}$ and $\mathbf{x} \in Q^{****}$ we have
\begin{align*}
    C^{-1}d(Q)^{-1} \leq m(\mathbf{x}) \leq Cd(Q)^{-1}.
\end{align*}
\end{fact}

\begin{proposition}[{\cite[Proposition 4.4]{Hejna}}]\label{propo:F}
Assume that $V \in {\rm{RH}}^{q}(dw)$, where $q>\max(1,\frac{\mathbf{N}}{2})$, and $V \geq 0$. The covering $\mathcal{Q}$ defined by~\eqref{eq:stopping_time} satisfies the following finite overlapping condition:
\begin{equation}\label{eq:finite_overlapping_new}
     \left(\exists C_0>0\right)\left(\forall Q_1,Q_2 \in \mathcal{Q}\right)\, Q_1^{****} \cap Q_2^{****}\neq \emptyset \Rightarrow C_0^{-1}d(Q_1) \leq d(Q_2) \leq C_0 d(Q_1).
\end{equation}
\end{proposition}

\begin{lemma}[{\cite[Lemma 5.3]{Hejna}}]\label{lem:general_Leibniz}
For all  $j \in \{1,2,\ldots,N\}$, $g \in C^\infty_{c}(\mathbb{R}^N)$, and $f \in L^2(dw)$ such that its weak Dunkl derivative $T_jf$ is in $L^2(dw)$ we have $T_j(fg) \in L^2(dw)$. Moreover,
\begin{align*}
    T_j(fg)(\mathbf{x})=(T_jf)(\mathbf{x})g(\mathbf{x})+f(\mathbf{x})\partial_{j}g(\mathbf{x})+\sum_{\alpha \in R}\frac{k(\alpha)}{2} \alpha_j f(\sigma_{\alpha}(\mathbf{x}))\frac{g(\mathbf{x})-g(\sigma_{\alpha}(\mathbf{x}))}{\langle \mathbf{x}, \alpha \rangle}
\end{align*}
in $L^2(dw)$-sense.
\end{lemma}

The following lemma is inspired by its counterpart for fractional laplacian~\cite[Lemma 9.6]{DZ_Studia}.

\begin{lemma}[{\cite[Lemma 5.5]{Hejna}}]\label{lem:gradient}
Assume that $V \in {\rm{RH}}^{q}(dw)$, where $q>\max(1,\frac{\mathbf{N}}{2})$, and $V \geq 0$. There is a constant $C>0$ such that for all $j\in \{1,\ldots,N\}$, $f \in L^2(dw)$ such that its weak Dunkl derivative $T_jf$ is in $L^2(dw)$, and $Q \in \mathcal{Q}$ we have
\begin{align*}
    \|T_j(f\phi_{Q})\|_{L^2(dw)} \leq C\left(\left(\int_{Q^{*}}|T_jf(\mathbf{x})|^2\,dw(\mathbf{x})\right)^{1/2}+\left(\int_{\mathcal{O}(Q^{*})}|f(\mathbf{x})|^2m(\mathbf{x})^2\,dw(\mathbf{x})\right)^{1/2}\right).
\end{align*}
\end{lemma}

The next theorem is a version of Fefferman--Phong inequality (\cite[p. 146]{Fefferman}, see also Shen ~\cite{Shen2},~\cite[Lemma 1.9]{Shen}). It is the main result of~\cite{Hejna}.
\begin{theorem}[Fefferman--Phong type inequality, see {\cite[Theorem 1.1]{Hejna}}]\label{theo:Fefferman-Phong}
Assume that $V \in {\rm{RH}}^{q}(dw)$, where $q>\max(1,\frac{\mathbf{N}}{2})$, and $V \geq 0$. There is a constant $C>0$, which depends on $R$, $k$, $N$, $q$, and $C_{{\rm RH}}$, such that for all $f \in \mathcal{D}(\mathbf{Q})$ we have
\begin{equation}\label{eq:F-Ph}
    \int_{\mathbb{R}^N}|f(\mathbf{x})|^2m(\mathbf{x})^2\,dw(\mathbf{x}) \leq C \mathbf{Q}(f,f).
\end{equation}
\end{theorem}

\section{H\"older bounds for Dunkl translation of radial function}

The next lemma is a version of~\cite[Theorem 4.1 (b)]{ADzH} with the Dunkl heat kernel replaced by the Dunkl translation of radial $C^{\infty}_{c}(\mathbb{R}^N)$-function. Its proof is similar to the proof of~\cite[Theorem 4.1 (b)]{ADzH} and it is based on R\"osler's formula~\eqref{eq:translation-radial}.
\begin{lemma}\label{lem:Psi_radial}
Let $\varphi \in C_c^{\infty}(\mathbb{R}^N)$ be a radial function supported by the unit ball, that is, $\varphi(\mathbf{x})=\widetilde{\varphi}(\|\mathbf{x}\|)$, where $\widetilde \varphi \in C_c^\infty(-1,1)$ is even.   There is a constant $C>0$ such that for all $\mathbf{x},\mathbf{y},\mathbf{z} \in \mathbb{R}^{N}$ and $t>0$ such that $\|\mathbf{y}-\mathbf{z}\| < t$ we have
\begin{align*}
&|\varphi_t(\mathbf{x},\mathbf{y})-\varphi_t(\mathbf{x},\mathbf{z})|\leq C\frac{\|\mathbf{y}-\mathbf{z}\|}{t}\Big(\max(w(B(\mathbf{x},t)),w(B(\mathbf{y},t)))\Big)^{-1}\chi_{[0,2]}(d(\mathbf{x},\mathbf{y})/t).
\end{align*}
\end{lemma}

\begin{proof}
The presence of the factor $\chi_{[0,2]}(d(\mathbf{x},\mathbf{y})/t)$ follows by Lemma~\ref{lem:nonradial_estimation}. For $s \in [0,1]$ we set $\mathbf{y}_s=\mathbf{z}+s(\mathbf{y}-\mathbf{z})$. By~\eqref{eq:translation-radial} we obtain
\begin{equation}\label{eq:Roesler-formula_1}
\begin{split}
   \varphi_t(\mathbf{x},\mathbf{y})-\varphi_t(\mathbf{x},\mathbf{z})&=t^{-\mathbf{N}}\int_{\mathbb{R}^N} \widetilde{\varphi}(A(\mathbf{x},\mathbf{y},\eta)/t)-\widetilde{\varphi}(A(\mathbf{x},\mathbf{z},\eta)/t)\,d\mu_{\mathbf{x}}(\eta)\\&=t^{-\mathbf{N}}\int_{\mathbb{R}^N}\,\int_0^1 \frac{d}{ds}\widetilde{\varphi}(A(\mathbf{x},\mathbf{y}_s,\eta))/t)\,ds\, d\mu_{\mathbf{x}}(\eta).
\end{split}
\end{equation}
Clearly, there is an even function $\widetilde{\boldsymbol \phi}\in C_c^\infty(-1,1)$, such that $|\widetilde{\varphi}'(x)|\leq \widetilde{\boldsymbol \phi}(x)$ for all $x \in \mathbb{R}$. 
Hence, by Cauchy--Schwarz inequality,~\eqref{A3}, and~\eqref{eq:Roesler-formula_1},  we obtain
\begin{equation}\label{eq:Roesler-formula_2}
    \begin{split}
        \Big|\frac{d}{ds}\widetilde{\varphi}(A(\mathbf{x},\mathbf{y}_s,\eta))/t)\Big|&=\big|\widetilde{\varphi}'(A(\mathbf{x},\mathbf{y}_s,\eta))/t)\big|\cdot t^{-1}\cdot \big|\langle \nabla_{\mathbf{y}_s}A(\mathbf{x},\mathbf{y}_s,\eta),2(\mathbf{y}-\mathbf{z})\rangle\big |
        \\&\leq C \frac{\|\mathbf{y}-\mathbf{z}\|}{t}{\widetilde{\boldsymbol \phi}(A(\mathbf{x},\mathbf{y}_s,\eta)/t)} .
    \end{split}
\end{equation}
Set ${\boldsymbol \phi}(\mathbf x)={\widetilde{\boldsymbol \phi}}(\| \mathbf x\|)$.  Combining~\eqref{eq:Roesler-formula_1} and~\eqref{eq:Roesler-formula_2} we obtain
\begin{equation}\label{eq:Roesler-formula_3}
    \begin{split}
        |\varphi_t(\mathbf{x},\mathbf{y})-\varphi_t(\mathbf{x},\mathbf{z})| &\leq Ct^{-\mathbf{N}} \frac{\|\mathbf{y}-\mathbf{z}\|}{t}\int_{\mathbb{R}^N} \int_0^1 {\widetilde{\boldsymbol \phi}(A(\mathbf{x},\mathbf{y}_s,\eta)/t)}\,ds\,d\mu_{\mathbf{x}}(\eta)\\
        &=C \frac{\|\mathbf{y}-\mathbf{z}\|}{t}\int_0^1{\boldsymbol \phi}_t(\mathbf x,\mathbf y_s)\,ds.
    \end{split}
\end{equation}
Finally, applying Lemma \ref{lem:nonradial_estimation} and using the assumption $\|\mathbf{y}-\mathbf{z}\|<t$ we get
\begin{align*}
    {\boldsymbol \phi}_t(\mathbf x,\mathbf y_s)\leq C\Big(1+\frac{\| \mathbf x-\mathbf y_s\|}{t}\Big)^{-2}\Big(\max(w(B(\mathbf{x},t)),w(B(\mathbf{y}_s,t)))\Big)^{-1}\chi_{[0,2]}(d(\mathbf{x},\mathbf{y}_s)/t), 
\end{align*}
which, together with~\eqref{eq:Roesler-formula_3} and~\eqref{eq:doubling}, gives the claim.
\end{proof}

\section{Proof of Theorem~\ref{teo:main_box}}

\begin{definition}\label{def:resolution}\normalfont
By the smooth resolution of identity $\{\phi_{Q}\}_{Q \in \mathcal{Q}}$ associated with $\mathcal{Q}$ (see~\eqref{eq:stopping_time}) we mean the collection of $C^{\infty}$-functions on $\mathbb{R}^N$ such that $\supp \phi_{Q} \subseteq Q^{*}$, $0 \leq \phi_Q(\mathbf{x}) \leq 1$,
\begin{equation}\label{eq:partition_derivative}
    |\partial^{\alpha}\phi_{Q}(\mathbf{x})| \leq C_{\alpha}d(Q)^{-|\alpha|} \text{ for all }\alpha \in \mathbb{N}_0^{N},
\end{equation}
and $\sum_{Q \in \mathcal{Q}}\phi_Q(\mathbf{x})=1$ for all $\mathbf{x} \in \mathbb{R}^N$. The collection $\{\phi_{Q}\}_{Q \in \mathcal{Q}}$ is well-defined thanks to Proposition~\ref{propo:F}.
\end{definition}

The proof of Theorem~\ref{teo:main_box} is partially based on~\cite[Theorem 4]{KurataSugano}.
\begin{proof}[Proof of the first inequality in~\eqref{eq:main_1}]
By the min--max principle it is enough to find $M(\lambda)$-dimensional subspace $\mathcal{H}$ of $L^2(dw)$ such that
\begin{equation}\label{eq:first_goal}
    \mathbf{Q}(u,u) \leq C_1 \lambda \|u\|_{L^2(dw)}^2 \text{ for all }u \in \mathcal{H}.
\end{equation}
Fix a small real number $\varepsilon \in (0,1)$ of the form $\varepsilon=2^{-s}$ for some $s \in \mathbb{N}$ (it will be chosen latter on). By the definition of $M(\lambda)$, there are at least $M(\lambda)$ cubes $K \in ({\rm Grid})_{\varepsilon\lambda^{-1/2}}$ such that $K \cap E_{\lambda} \neq \emptyset$. For any cube $K \in ({\rm Grid})_{\varepsilon\lambda^{-1/2}}$ such that $K \cap E_{\lambda} \neq \emptyset$ let $\eta_{K}$ be a nonzero smooth function supported by $K$ such that
\begin{equation}\label{eq:gradient_eta}
    |\partial^{\alpha}\eta_{K}(\mathbf{x})| \leq C_{\alpha,\varepsilon}\lambda^{\frac{|\alpha|+\mathbf{N}}{2}} \text{ for all }\alpha \in \mathbb{N}_0^{N}, \, \mathbf{x} \in \mathbb{R}^N,
\end{equation}
\begin{equation}\label{eq:eta_upper}
    |\eta_{K}(\mathbf{x})| \geq C\lambda^{\frac{\mathbf{N}}{2}} \text{ for all }\mathbf{x} \in K_{*}
\end{equation}
(here $K_{*}$ denotes the cube of the same center as $K$ but two times smaller side-length). Since the interiors of the cubes from $({\rm Grid})_{\varepsilon\lambda^{-1/2}}$ are pairwise disjoint, it is enough to check that each of $\eta_{K}$ satisfies~\eqref{eq:first_goal}. 
It is the standard fact that
\begin{equation}\label{eq:der_compare}
\|T_jf\|_{L^{\infty}} \leq C\|\partial_{j}f\|_{L^{\infty}} \text{ for all }j \in \{1,\ldots,N\}\text{ and }f \in C^{1}(\mathbb{R}^N)
\end{equation}
(see e.g.~\cite[Lemma 5.5]{Hejna} for details). Therefore, by~\eqref{eq:der_compare} and~\eqref{eq:gradient_eta} we obtain
\begin{equation}\label{eq:upper_der}
    \sum_{j=1}^{N}\int_{K}|T_{j}\eta_{K}(\mathbf{x})|^2\,dw(\mathbf{x}) \leq C_{\varepsilon} \lambda^{\mathbf{N}+1}w(K).
\end{equation}
Let $\mathbf{x}_{K}$ be the center of $K$. Since $K \cap E_{\lambda} \neq \emptyset$, by Lemma~\ref{lem:m_growth} we get $\varepsilon\sqrt{N}\lambda^{-1/2} \leq m(\mathbf{x}_K)^{-1}$ for $\varepsilon$ small enough. Consequently, by the doubling property of the  measure $dw$ (see~\eqref{eq:doubling}), Lemma~\ref{lem:Rr}, and the definition of $m$ (see~\eqref{eq:m}) we obtain
\begin{equation}\label{eq:upper_pot}
\begin{split}
    &\int_{K}V(\mathbf{x})\eta_{K}(\mathbf{x})^2\,dw(\mathbf{x})\\& \leq C\lambda^{\mathbf{N}+1}w(K)\frac{\lambda^{-1}}{w(B(\mathbf{x}_K,\frac{1}{2}\varepsilon\lambda^{-1/2}))}\int_{B(\mathbf{x}_K,\varepsilon\sqrt{N}\lambda^{-1/2})}V(\mathbf{x})\,dw(\mathbf{x}) \\&\leq C_{\varepsilon}\lambda^{\mathbf{N}+1}w(K)\frac{m(\mathbf{x}_K)^{-2}}{w(B(\mathbf{x}_K,m(\mathbf{x}_K)^{-1}))}\left(\frac{\varepsilon\sqrt{N}\lambda^{-1/2}}{m(\mathbf{x}_K)^{-1}}\right)^{\gamma}\int_{B(\mathbf{x}_K,m(\mathbf{x}_K)^{-1})}V(\mathbf{x})\,dw(\mathbf{x}) \\&\leq C_{\varepsilon}\lambda^{\mathbf{N}+1}w(K). 
\end{split}
\end{equation}
By~\eqref{eq:upper_der} and~\eqref{eq:upper_pot} we get
\begin{equation}\label{eq:first_final_1}
    \mathbf{Q}(\eta_{K},\eta_{K}) \leq C_{\varepsilon}\lambda^{\mathbf{N}+1}w(K).
\end{equation}
On the other hand, by~\eqref{eq:growth} and~\eqref{eq:eta_upper} we have
\begin{equation}\label{eq:first_final_2}
    \lambda^{\mathbf{N}+1}w(K) \leq C\lambda\int_{K}\eta_{K}(\mathbf{x})^2\,dw(\mathbf{x}).
\end{equation}
Finally,~\eqref{eq:first_goal} follows by~\eqref{eq:first_final_1} and~\eqref{eq:first_final_2}.
\end{proof}

\begin{proof}[Proof of the second inequality in~\eqref{eq:main_1}]
By the min-max principle, it suffices to show the existence of  a subspace $\mathcal{H}$ of $L^2(dw)$ satisfying the following conditions: there exist constants $C_2,C_3>0$ such that
\begin{equation}
    \dim \mathcal{H} \leq C_2M(\lambda),
\end{equation}
\begin{equation}\label{eq:perpendicular}
    \mathbf{Q}(u,u) \geq C_3\lambda\|u\|_{L^2(dw)}^2\text{ for all }u \perp \mathcal{H} \text{ and }u\in \mathcal{D}(\mathbf{Q}).
\end{equation}

Let $\Psi \in C^{\infty}_{c}(\mathbb{R}^N)$ be a radial function such that $\int_{\mathbb{R}^N}\Psi(\mathbf{x})\,dw(\mathbf{x})=1$ and $\supp \Psi \subseteq B(0,1)$. It follows from Corollary~\ref{coro:Roesler} that
\begin{equation}\label{eq:integral_one}
    |\mathcal{F}\Psi(\xi)-1| \leq C\|\xi\| \text{ for all }\xi \in \mathbb{R}^N.
\end{equation} 
Set
\begin{align*}
    \Psi^{\lambda}(\mathbf{x})=\lambda^{\mathbf{N}/2}\Psi(\lambda^{1/2}\mathbf{x}).
\end{align*}
Let us consider $Q \in \mathcal{Q}$. If $Q^{***} \cap E_{\lambda}^c \ne \emptyset$, then thanks to Fact~\ref{fact:m_d_compare} we have $m(\mathbf{x})>c\sqrt{\lambda}$ for all $\mathbf{x} \in Q^{*}$. Consequently, for any $u \in \mathcal{D}(\mathbf{Q})$ we have
\begin{equation}\label{eq:outside}
    \lambda \int_{Q^{*}}|(u\phi_{Q})(\mathbf{x})|^2\,dw(\mathbf{x}) \leq  c^{-2}\int_{Q^{*}}|u(\mathbf x)|^2 m(\mathbf{x})^2\,dw(\mathbf{x}).
\end{equation}
If $Q^{***} \cap E_{\lambda}^c = \emptyset$, then 
$Q^{***}\subseteq E_\lambda$, so  
$m(\mathbf{x})\leq  \sqrt{\lambda}$ for all $\mathbf{x} \in Q^{***} $, and, by Fact~\ref{fact:m_d_compare}, $d(Q)\geq c\lambda^{-1/2}$.  
For such a cube $Q$ we write

\begin{equation}\label{eq:S1S2}
\begin{split}
    &\lambda \int_{Q^{*}}|(u\phi_{Q})(\mathbf{x})|^2\,dw(\mathbf{x})\\
    &\leq C\lambda  \int_{Q^*} |(u\phi_{Q})(\mathbf{x}) - \Psi^\lambda  *(u\phi_Q)(\mathbf{x})|^2\,dw(\mathbf{x}) +C\lambda \int_{Q^*} |\Psi^\lambda * (u\phi_Q)(\mathbf{x})|^2\,dw(\mathbf{x})\\&=:S_1+S_2.  
\end{split}
\end{equation}
By  Plancherel's formula~\eqref{eq:Plancherel},~\eqref{eq:der_transform},  and~\eqref{eq:integral_one} we have
\begin{equation}\label{eq:diff}
    \begin{split}
        S_1 & \leq C\lambda \int_{\mathbb R^N} \Big|\mathcal F(u\phi_Q)(\xi)\Big(1-\mathcal F(\Psi^{\lambda})(\xi)\Big)\Big|^2\, dw(\xi )\\
        &\leq C\lambda \int_{\mathbb R^N} |\mathcal F(u\phi_Q)(\xi)|^2 \lambda^{-1} \|\xi\|^2\, dw(\xi)\\
        &\leq C \int_{\mathbb R^N} \sum_{j=1}^N |T_j(u\phi_Q)(\mathbf x)|^2dw(\mathbf x).  
    \end{split}
\end{equation}
The first inequality in \eqref{eq:diff} can be thought as a counterpart of the pseudo-Poincar\'e inequality (see~\cite{DZ_Studia},~\cite{Velicu1}). Using Lemma \ref{lem:gradient} we get 
\begin{equation}\label{eq:S1}
S_1\leq C\sum_{j=1}^N \int_{Q^*}|T_ju(\mathbf x)|^2\, dw(\mathbf x)+ C\int_{\mathcal O(Q^*)} |u(\mathbf x)|^2m(\mathbf x)^2\, dw(\mathbf x).
\end{equation}
Fix a small real number $\varepsilon >0$ (it will be chosen latter on). Let 
$$ \beta(Q)=\{ K\in {\rm (Grid)}_{\varepsilon \lambda^{-1/2}}: K\cap Q^*\ne \emptyset\}.$$ 
Let $\mathbf{x}_K$ denote the center of $K\in\beta(Q)$. Set 
$$\mathcal H_Q={\rm span}\{ \Psi^{\lambda}(\mathbf{x}_K, \cdot)\phi_Q(\cdot): K\in \beta(Q)\}.$$
Clearly, $${\rm dim}\, \mathcal H_Q\leq C_N \varepsilon^{-N}{\rm   card}\,\{K\in {\rm (Grid)}_{\lambda^{-1/2}}\cap Q^*\ne\emptyset \}.$$ Then, by the definion of the Dunkl convolution (see~\eqref{eq:conv_def}), for $u\perp \mathcal H_Q$ we have  
\begin{equation}\begin{split}
    S_2&\leq C\lambda \int_{Q^*} \Big| \int_{Q^*} \sum_{K\in\beta(Q)} \chi_K(\mathbf x) \Psi^{\lambda} (\mathbf x,\mathbf y)\phi_Q(\mathbf y)u(\mathbf y)\,dw(\mathbf{y})\Big|^2\, dw(\mathbf x)\\
    &\leq  C\lambda \int_{Q^*} \Big| \int_{Q^*} \sum_{K\in\beta(Q)} \chi_K(\mathbf x)\Big( \Psi^{\lambda} (\mathbf x,\mathbf y)- \Psi^{\lambda}(\mathbf x_K,\mathbf y)\Big)\phi_Q(\mathbf y)u(\mathbf y)\,dw(\mathbf{y})\Big|^2\, dw(\mathbf x).\\
    \end{split}
 \end{equation}
 Consider the integral kernel $K_Q(\mathbf x,\mathbf y)=\sum_{K\in\beta (Q)}\chi_K(\mathbf x)| \Psi^{\lambda} (\mathbf x,\mathbf y)- \Psi^{\lambda} (\mathbf x_K,\mathbf y)| $. Then, for fixed $\mathbf x\in Q^*$, let $K'$ be the unique one such that  $\mathbf x\in K'\in\beta(Q)$. So, by Lemma \ref{lem:Psi_radial}, we have 
 \begin{equation}\label{eq:KQ1} \int K_Q(\mathbf x,\mathbf y)\, dw(\mathbf y)=\int |\Psi^{\lambda} (\mathbf x,\mathbf y)- \Psi^{\lambda} (\mathbf x_{K'},\mathbf y)|\, dw(\mathbf y)\leq C \varepsilon.  \end{equation}
 Now fix $\mathbf y\in Q^*$.   Applying once more Lemma  \ref{lem:Psi_radial}, we obtain 
 \begin{equation}
     \begin{split}
         K_Q(\mathbf x,\mathbf y)&\leq C \sum_{K\in\beta (Q)} \chi_K(\mathbf x)\frac{\| \mathbf x-\mathbf x_K\| \sqrt{\lambda}}{  w(B(\mathbf y,\lambda^{-1/2}))} \chi_{[0,2]}(d(\mathbf x,\mathbf y) \lambda^{1/2})\\
         &\leq  \frac{C \varepsilon}{  w(B(\mathbf y,\lambda^{-1/2}))}\chi_{[0,2]}(d(\mathbf x,\mathbf y)\lambda^{1/2}).
     \end{split}
 \end{equation}
 Consequently, 
 \begin{equation}\label{eq:KQ2} \int K_Q(\mathbf x,\mathbf y)\, dw(\mathbf x)\leq C \varepsilon . 
 \end{equation}
 Finally, thanks to Schur's test, \eqref{eq:KQ1}, and \eqref{eq:KQ2} we obtain 
 
 \begin{equation}\label{eq:S2_final}
     S_2\leq C\lambda \varepsilon^2 \| u\phi_Q\|^2_{L^2(dw)}.
 \end{equation}
 
Note that $\dim \left(\bigoplus_{Q^{***} \subseteq E_{\lambda}}\mathcal{H}_{Q}\right) \leq C_2 M(\lambda)$. Now, if $u$ is orthogonal to the all Hilbert spaces $\mathcal H_Q$ for $Q \in \mathcal{Q}$ such that $Q^{***}\subseteq E_{\lambda}$, by~\eqref{eq:S1S2},~\eqref{eq:S1},~\eqref{eq:S2_final}, and Proposition~\ref{propo:F}, we conclude 
 \begin{equation*}\begin{split}
     \lambda \| u\|_{L^2(dw)}^2 
     &\leq C\lambda \sum_{Q^{***}\cap E_\lambda^c\ne \emptyset}\| u\phi_Q\|_{L^2(dw)}^2+ 
     C\lambda \sum_{Q^{***}\subseteq E_{\lambda} }
     \| u\phi_Q\|_{L^2(dw)}^2 \\
     &\leq C\int_{\mathbb R^N} |u(\mathbf x)|^2m(\mathbf x)^2\, dw(\mathbf x) +   C \sum_{j=1}^N \int_{\mathbb R^N}|T_ju(\mathbf x)|^2\, dw(\mathbf x) +C\lambda \varepsilon^2 \| u\|_{L^2(dw)}^2. 
 \end{split}
 \end{equation*}
 Now taking $\varepsilon$ small enough and using the Fefferman--Phong inequality (see Theorem~\ref{theo:Fefferman-Phong}) we obtain the claim.  
\end{proof}

\begin{proof}[Proof of~\eqref{eq:min_m}]
Let $f_0 \in \mathcal{D}(L)$ be a nonzero function such that $Lf_{0}=\lambda_{0}f_{0}$. Thanks to the Fefferman--Phong inequality we have 
\begin{equation}
\begin{split}
    \min_{\mathbf{x} \in \mathbb{R}^N}m(\mathbf{x})^2 \|f_0\|_{L^2(dw)}^2 &\leq \int_{\mathbb{R}^N}|f_0(\mathbf{x})|^2m(\mathbf{x})^2\,dw(\mathbf{x}) \leq C\mathbf{Q}(f_0,f_0)\\&=C\langle Lf_0,f_0\rangle=C\lambda_0^2 \|f_0\|_{L^2(dw)}^{2},
\end{split}
\end{equation}
so~\eqref{eq:min_m} follows.
\end{proof}

{\bf Acknowledgment.} The author would like to thank Jacek Dziuba\'nski for his helpful comments and suggestions.

\end{document}